\documentclass[review]{elsarticle}

\usepackage{lineno,hyperref}
\modulolinenumbers[5]

\journal{Journal of \LaTeX\ Templates}

\usepackage{amsmath,amssymb,latexsym,amsmath,amsthm,amscd}
\usepackage{setspace}

\theoremstyle{plain}
\newtheorem{thm}{Theorem}
\newtheorem{lem}[thm]{Lemma}
\newtheorem{prop}[thm]{Proposition}
\newtheorem{cor}[thm]{Corollary}

\theoremstyle{definition}
\newtheorem{dfn}[thm]{Definition}
\newtheorem{ex}[thm]{Example}

\theoremstyle{remark}
\newtheorem{rmk}[thm]{Remark}

\DeclareMathOperator{\End}{End}

\begin{document}

\begin{frontmatter}

\title{Remarks on Mathieu-Zhao Subspaces of Commutative Associative Algebras and Vertex Algebras}

\author{Gaywalee Yamskulna}

\address[mymainaddress]{Department of Mathematics\\ Illinois State University, Normal, IL}

\begin{abstract}
We introduce a notion of Mathieu-Zhao subspaces of vertex algebras. Among other things, we show that for a vertex algebra $V$ and its subspace $M$ that contains $C_2(V)$, $M$ is a Mathieu-Zhao subspace of $V$ if and only if the quotient space $M/C_2(V)$ is a Mathieu-Zhao subspace of a commutative associative algebra $V/C_2(V)$. As a result,  one can study the famous Jacobian conjecture in terms of Mathieu-Zhao subspaces of vertex algebras. In addition, for a $CFT$-type vertex operator algebra $V$ that satisfies the $C_2$-cofiniteness condition, we classify all Mathieu-Zhao subspaces $M$ that contain $C_2(V)$. 
\end{abstract}

\begin{keyword}
\texttt{Vertex operator algebras}
\MSC[2010] 17B69
\end{keyword}

\end{frontmatter}


\section{Introduction}

The notion of Mathieu-Zhao subspaces of commutative associative algebras is a generalization of the fundamental notion of ideals and was introduced by Zhao in 2009 (cf. \cite{Z1}). The study of Mathieu-Zhao subspaces of commutative associative algebras grew out of many attempts to understand the Jacobian conjecture (cf. \cite{M, ZI, Z2, Z4}). This conjecture is a famous problem on polynomials in several variables and was proposed by Keller in 1939. The Jacobian conjecture claims that any polynomial map $F$ of $\mathbb{C}^n$ with $\det J(F)=1$ is an automorphism of $\mathbb{C}^n$ (cf. \cite{K}). Here $J(F)=\left(\frac{\partial F(x_i)}{\partial x_j}\right)_{1\leq i,j\leq n}$. This conjecture remains open even for the case when $n=2$. In 1998, the Fields Medalist and the Wolf Prize Winner S. Smale included the Jacobian conjecture in his list of 18 fundamental mathematical problems for the 21st century (cf. \cite{S}).

The theory of vertex (operator) algebras has been showing its powers in the solution of concrete mathematical problems and in the understanding of conceptual but subtle mathematical and physical structures of conformal field theories. The original motivation for the introduction of the notion of vertex operator algebra arose from the problem of realizing the Monster group as a symmetry group of a certain infinite-dimensional graded vector space with natural additional structure (cf. \cite{B1, B2, FLM1, FLM2}). Vertex (operator) algebras are known as the mathematical counterparts of chiral algebras in two-dimensional quantum conformal field theory (cf. \cite{BPZ, MS}). In many aspects, vertex (operator) algebras are analogous to commutative associative algebras. In fact, commutative vertex algebras are commutative associative algebras with derivations (cf. \cite{B1}). 

The aim of this paper is to deepen our understanding on connections between vertex (operator) algebras and commutative associative algebras through Mathieu-Zhao subspaces of commutative associative algebras. In addition, we want to find a way to study the Jacobian conjecture in term of vertex algebras. In section 2, we provide necessary background on vertex algebras $V$ and commutative associative algebras $V/C_2(V)$, which is an important algebra for the study of representation theory of vertex operator algebras, that we will need in later sections. In section 3, we introduce a notion of Mathieu-Zhao subspaces of vertex algebras and study properties of these subspaces. We show that for a vertex algebra $V$ and its subspace $M$ that contains $C_2(V)$, $M$ is a Mathieu-Zhao subspace of $V$ if and only if the quotient space $M/C_2(V)$ is a Mathieu-Zhao subspace of a commutative associative algebra $V/C_2(V)$. In section 4, we study Mathieu-Zhao subspaces of vertex algebras of $CFT$-type that satisfy the $C_2$-cofinite condition. Particularly, we show that if $M$ is a subspace of $V$ that contains $C_2(V)$ and ${\bf 1}\not\in M$, then $M$ is a Mathieu-Zhao subspace of $V$. In addition, we discuss a connection between the Jacobian conjecture and certain subspaces of Heisenberg vertex operator algebras.

\section{Background Materials}

In this section we will review materials on vertex algebras $V$, and commutative associative algebras $V/C_2(V)$ that we will need in later sections. Also, we recall examples of vertex algebras that we will refer to in later sections.

\begin{dfn}(\cite{B1, FLM2, LLi} )A {\em vertex algebra} is a vector space $V$ equipped with a linear map 
\begin{eqnarray}
Y:V&\rightarrow&(\End V)[[z,z^{-1}]]\nonumber\\ 
v&\mapsto&Y(v,z)=\sum_{n\in\mathbb{Z}}v_n z^{-n-1}\text{ (where $v_n\in\End V$)}
\end{eqnarray}
and equipped with a distinguised vector ${\bf 1}$, called the {\em vacuum (vector)}, such that for $u,v\in V$, 
\begin{eqnarray}
&&u_nv=0\text{ for $n$ sufficiently large},\\
&&Y({\bf 1},z)=1,\\
&&Y(v,z){\bf 1}\in V[[z]]\text{ and }\lim_{z\rightarrow 0}Y(v,z){\bf 1}=v
\end{eqnarray}
and such that 
\begin{eqnarray}
&&z_0^{-1}\delta\left(\frac{z_1-z_2}{z_0}\right)Y(u,z_1)Y(v,z_2)-z_0^{-1}\delta\left(\frac{z_2-z_1}{-z_0}\right)Y(v,z_2)Y(u,z_1)\nonumber\\
=&&z_2^{-1}\delta\left(\frac{z_1-z_0}{z_2}\right)Y(Y(u,z_0)v,z_2)
\end{eqnarray}
the {\em Jacobi identity}. 
\end{dfn}

\smallskip

A vertex algebra $V$ equipped with a $\mathbb{Z}$-grading $V=\coprod_{n\in\mathbb{Z}} V_{(n)}$ is called a $\mathbb{Z}$-graded vetex algebra if ${\bf 1}\in V_{(0)}$ and if for $u\in V_{(k)}$ with $k\in\mathbb{Z}$ and for $m,n\in\mathbb{Z}$, 
\begin{equation}
u_mV_{(n)}\subseteq V_{(k+n-m-1)}.
\end{equation} 
An $\mathbb{N}$ graded vertex algebra is defined in the obvious way.

\smallskip

From the Jacobi identity we have Borcherds' commutator formula and iterate formula:
\begin{eqnarray}
&&{[u_m,v_n]}=\sum_{i\geq 0}{m\choose i}(u_iv)_{m+n-i},\label{commutativity}\\
&&(u_mv)_nw=\sum_{i\geq 0}(-1)^i {m\choose i}(u_{m-i}v_{n+i}w-(-1)^mv_{m+n-i}u_iw)\label{associativity}
\end{eqnarray}
for $u,v,w\in V$, $m,n\in\mathbb{Z}$. Specifically, we have
\begin{eqnarray}
&&[u_{-1},v_{-1}]=\sum_{i\geq 0}(-1)^i(u_iv)_{-2-i}\\
&&(u_0v)_0=u_0v_0-v_0u_0\\
&&(u_0v)_{-1}=u_0v_{-1}-v_{-1}u_0\\
&&(u_{-1}v)_{-1}=\sum_{i\geq 0} u_{-1-i}v_{-1+i}+v_{-2-i}u_i.
\end{eqnarray}

Define a linear operator $\mathcal{D}$ on $V$ by 
\begin{equation}
\mathcal{D}(v)=v_{-2}{\bf 1}\text{ for }v\in V. 
\end{equation}
We have \begin{equation}\label{Drel}
[\mathcal{D},v_n]=(\mathcal{D}v)_n=-nv_{n-1}\text{ for }v\in V, n\in\mathbb{Z}.
\end{equation}
Moreover, $Y(u,z)v=e^{z\mathcal{D}}Y(v,-z)u$ for $u,v\in V$. Equivalently, 
$$u_nv=-\sum_{i\geq 0}\frac{(-1)^{i+n}}{i!}\mathcal{D}^iv_{i+n}u.$$ 
In particular, we have 
\begin{eqnarray}
u_0v&&=-v_0u-\sum_{i\geq 1}\frac{(-1)^{i}}{i!}\mathcal{D}^iv_{i}u,\label{u0v}\\
u_0u&&=\frac{1}{2}\sum_{i=1}^{\infty}\frac{D^i}{i!}u_iu(-1)^{-i-1}\in D(V),\label{u0u}\\
u_{-1}v&&=v_{-1}u-\sum_{i\geq 1}\frac{(-1)^{i+1}}{i!}\mathcal{D}^iv_{i-1}u,\label{u-1v}\\
u_0\mathcal{D}(v)&&=\mathcal{D}(u(0)v)\text{ and }u_{-1}\mathcal{D}(v)=\mathcal{D}(u_{-1}v)-u_{-2}v.
\end{eqnarray}

\smallskip

\begin{dfn}(\cite{FLM2, LLi}) A vertex operator algebra is a $\mathbb{Z}$-graded vertex algebra $$V=\coprod_{n\in\mathbb{Z}} V_{(n)}\text{ for }v\in V_{(n)},~n={\it wt } v,$$ such that $\dim V_{(n)}<\infty$ for $n\in\mathbb{Z}$ and $V_n=0$ for $n$ sufficiently negative, equipped with a distinguished homogeneous vector $\omega\in V_{(2)}$ (the conformal vector), satisfying the following conditions: 
$$[L(m),L(n)]=(m-n)L(m+n)+\frac{1}{12}(m^3-m)\delta_{m+n,0}c_V$$ for $m,n\in\mathbb{Z}$ where $Y(\omega,z)=\sum_{n\in\mathbb{Z}}L(n)z^{-n-2}$ and $c_V\in \mathbb{C}$. We also have $L(0)v=nv =({\it wt }v)v$ for $n\in\mathbb{Z}$ and $v\in V_{(n)}$; and $Y(L(-1)v,z)=\frac{d}{dz}Y(v,z)$.
\end{dfn}

\smallskip

\begin{dfn} A vertex operator algebra $V$ is of $CFT$-type if $V=\coprod_{n\geq 0} V_{(n)}$ and $V_0=\mathbb{C}{\bf 1}$.
\end{dfn}
\smallskip

\begin{dfn}(\cite{LLi}) An ideal of the vertex algebra $V$ is a subspace $I$ such that for all $v\in V$ and $w\in I$, $v_nw\in I$ and $w_nv\in I$ for all $v\in V$, $w\in I$ and $n\in\mathbb{Z}$.
\end{dfn}

\noindent If $V\neq \{0\}$ and $V$ is the only nonzero ideal we say that $V$ is simple.

\smallskip

\begin{dfn} Let $(V,Y,{\bf 1})$ be a vertex algebra. We define 
$$C_2(V):=Span_{\mathbb{C}}\{u_{-2}v~|~u,v\in V\}.$$ 
$V$ satisfies the $C_2$-cofiniteness condition if the dimension of $V/C_2(V)$ is finite. 
\end{dfn}

\smallskip

\begin{rmk}\label{dv} \ \ 

\begin{enumerate} 
\item $\mathcal{D}(V)\subseteq C_2(V)$, and $u_{-n}v\in C_2(V)$ for all $u,v\in V$ and $n\geq 2$. 
\item For $a\in V$, $w\in C_2(V)$, we have $a_0w$, $a_{-1}w\in C_2(V)$. 
\end{enumerate}
\end{rmk}

\smallskip

\begin{prop}\label{c2v}(\cite{Zh}) For a vertex algebra $V$, the $(0)^{th}$ product and $(-1)^{th}$ product induce a commutative Poisson algebra structure on $V/C_2(V)$. In particular, we have 
\begin{eqnarray*}
(a_0b)_0v&=&a_0(b_0v)-b_0(a_0v),\\
a_0b&\equiv&-b_0a\mod C_2(V),\\
(a_{-1}b)_{-1}v&\equiv&a_{-1}(b_{-1}v)\mod C_2(V),\\
(a_{-1}b)_0v&\equiv&a_{-1}(b_0v)+b_{-1}(a_0v)\mod C_2(V)
\end{eqnarray*} 
for $a,b,v\in V$. 
\end{prop}

\smallskip

\begin{cor}\label{a0-1a} Let $a\in V$. We have 
\begin{enumerate}
\item $a_0a_0....a_0a\in C_2(V)$.
\item Let $n_1,..., n_t\in\{0,-1\}$ such that the number of $n_i$ that equals to $-1$ is $k$ and $k$ is strictly less than $t$. Then 
$$a_{n_1}a_{n_2}...a_{n_t}a\equiv \underbrace{a_{-1}....a_{-1}}_{k\text{ terms}}a_0a_0...a_0a\mod C_2(V)\equiv 0\mod C_2(V).$$
\end{enumerate}
\end{cor} 
\begin{proof} These follow immediately from (\ref{u0u}), Remark \ref{dv}, and Proposition \ref{c2v}.
\end{proof}

\vspace{0.5cm}

\subsection{Examples of Vertex Algebras}

\smallskip

\begin{ex} {\bf Commutative Vertex Algebras}

\smallskip

A commutative vertex algebra is a vertex algebra $V$ such that $$[Y(u,z_1),Y(v,z_2)]=0$$ for all $u,v\in V$. Equivalently, $V$ is commutative if and only if $u_{n}v=0$ for all $u,v\in V$, $n\geq 0$. If $V$ is a commutative vertex algebra then $V$ is an honest commutative associative algebra with the product defined by $u\cdot v=u_{-1}v$ for $u,v\in V$, and with $\bf{1}$ as the identity element. Furthermore, the operator $\mathcal{D}$ of $V$ is a derivation and $Y(u,z)v=(e^{z\mathcal{D}}u )v$ for $u,v\in V$.

Also, for any commutative associative algebra $A$ with identity 1 equipped with a derivation $\partial$. $A$ carries the structure of a commutative vertex algebra where $Y$ is defined by $Y(a,z)=(e^{z\partial}a)b$ for $a,b\in A$. 
\end{ex}

\smallskip

\begin{ex} {\bf Heisenberg Vertex Operator Algebras} 

Let $\mathfrak{h}$ be a vector space equipped with a symmetric nondegenerate bilinear form $(\cdot,\cdot)$. We will identity $\mathfrak{h}$ with its dual $\mathfrak{h}^*$. The corresponding affine Lie algebra $$\hat{\mathfrak{h}}=\mathfrak{h}\otimes \mathbb{C}[t,t^{-1}]\oplus\mathbb{C}c$$ has a Lie structure defined by 
$$[x\otimes t^m,y\otimes t^n]=(x,y)m\delta_{m+n,0}c \text{ for }x,y\in\mathfrak{h}, m,n\in\mathbb{Z},\text{ and }[c,\hat{\mathfrak{h}}]=0.$$
Consider the induced $\hat{\mathfrak{h}}$-module 
$$M(1)=U(\hat{\mathfrak{h}})\otimes_{U(\mathfrak{h}\otimes \mathbb{C}[t]\oplus\mathbb{C}c)}\mathbb{C}\cong S(\hat{\mathfrak{h}}^-)\text{ (linearly)},$$ 
$h\otimes \mathbb{C}[t]$ acting triviall on $\mathbb{C}$ and $c$ acting as 1. For $\alpha\in\mathfrak{h}$ and $n\in\mathbb{Z}$, we write $\alpha(n)$ for the operator $\alpha\otimes t^n$ and set $\alpha(z)=\sum_{n\in\mathbb{Z}}\alpha(n)z^{-n-1}$. We define $$Y(v,z)={ }^{\circ}_{\circ}\left(\frac{1}{(n_1-1)!}\left(\frac{d}{dz}\right)^{n_1-1}\alpha_1(z)\right)...\left(\frac{1}{(n_k-1)!}\left(\frac{d}{dz}\right)^{n_k-1}\alpha_k(z)\right){ }^{\circ}_{\circ}$$ where $v=\alpha_1(-n_1)....\alpha_k(-n_k)\otimes 1\in M(1)$ for $\alpha_1,...,\alpha_k\in\mathfrak{h}$, $n_1,...,n_k$ are positive integers and where we use a normal ordering procedure, indicated by open colons, which signify that the expression is to be reordered if necessary so that all the operators $\alpha(n)$  ($\alpha\in\mathfrak{h},n<0$) are to be placed to the left of all the operator $\alpha(n)$ ($\alpha\in\mathfrak{h}, n\geq 0$) before the expression is evaluated . Extend this definition to all $v\in M(1)$ by linearity. Set ${\bf 1}=1$, $\omega=\frac{1}{2}\sum_{i=1}^d\beta^i(-1)^2\otimes 1\in M(1)$ where $\{\beta^i~|~1\leq i\leq d\}$ is an orthonormal basis of $\mathfrak{h}$. Then 
\begin{prop}(\cite{FLM2}) The space $M(1)$ is a simple vertex operator algebra. Moreover, $M(1)$ does not satisfy the $C_2$-condition. If $\dim\mathfrak{h}=d$, then $M(1)/C_2(V)\cong \mathbb{C}[x_1, ....,x_d]$ as commutative associative algebra. \end{prop}

\begin{rmk} The vertex operator algebra $M(1)$ is generated by $\{\beta_i(-1){\bf 1}~|~1\leq i\leq d\}$. When we identify $M(1)$ with a symmetric algebra $\mathbb{C}[x_{i,n}~|~i=1,...,d,~n=1,2,3,.....]$, it is easy to see that all vertex operators are built from the operators $x_{i,n}$ and $\frac{\partial}{\partial x_{i,n}}$.
\end{rmk}
\end{ex}

\smallskip

\begin{ex} {\bf Lattice Vertex Operator Algebras} 

Let $L$ be a finite-rank free abelian group equipped with a symmetric nondegenerate bilinear form $(~,~)$ such that $(\alpha,\alpha)\in 2\mathbb{Z}_+$ for all $\alpha\neq 0\in L$. Here $\mathbb{Z}_+$ is a set of positive integers. We set $\mathfrak{h}=\mathbb{C}\otimes_{\mathbb{Z}}L$ and extend the $\mathbb{Z}$-form $(\cdot, \cdot)$ of $L$ to $\mathfrak{h}$. 

Let $\hat{L}$ be the canonical central extension of $L$ by the cyclic group $\langle \pm 1\rangle$: $1\rightarrow\langle\pm 1\rangle\rightarrow \hat{L}\stackrel{-}{\rightarrow}L\rightarrow 1$ with the commutator map $c(\alpha,\beta)=(-1)^{(\alpha,\beta)}$ for $\alpha,\beta\in L$. Form the induced $\hat{L}$-module $\mathbb{C}\{L\}=\mathbb{C}[\hat{L}]\otimes_{\langle\pm 1\rangle}\mathbb{C}~(\cong \mathbb{C}[L]\text{ (linearly)}),$ where $\mathbb{C}[~\cdot~ ]$ denotes the group algebra and $-1$ acts on $\mathbb{C}$ as multiplication by $-1$. 

The underlying vector space for the vertex operator algebra $V_L$ is $$V_L=M(1)\otimes_{\mathbb{C}}\mathbb{C}\{L\}(\cong S(\hat{\mathfrak{h}}^-)\otimes \mathbb{C}[L] \text{ (linearly)}).$$ 

\begin{prop}(\cite{FLM2, DLM}) $V_L$ is a simple vertex operator algebra. Moreover, $V_L$ satisfies the $C_2$-condition. 
\end{prop} 
\end{ex}

\smallskip

\begin{ex} {\bf Vertex Operator Algebras associated with the Highest Weight Representations of Affine Lie Algebras}

Let $\mathfrak{g}$ be a finite-dimensional simple Lie algebra, ${\bf h}$ its Cartan subalgebra, and $\langle\cdot,\cdot\rangle$ normalized killing form such that the square length of a long root is 2. Let $\hat{\mathfrak{g}}=\mathfrak{g}\otimes \mathbb{C}[t,t^{-1}]\oplus\mathbb{C} c$ be the affine Lie algebra. Let $l$ be a complex number such that $l\neq -h\check{ }$ where $h\check{ }$ is the dual Coxeter number of $\mathfrak{g}$. Let $\mathbb{C}_l$ be the one-dimensional $(\mathfrak{g}\otimes \mathbb{C}[t]+\mathbb{C}c)$-module on which $c$ acts as scalar $l$ and $\mathfrak{g}\otimes \mathbb{C}[t]$ acts as zero. Form the generalized Verma $\hat{\mathfrak{g}}$-module $M_{\mathfrak{g}}(l,0)=U(\hat{\mathfrak{g}})\otimes_{U(\mathfrak{g}\otimes \mathbb{C}[t]+\mathbb{C}c)}\mathbb{C}_l$. 

\begin{prop}(\cite{DL, DLM, FZ}) \ \ 

\begin{enumerate}
\item If $l\neq -h\check{ }$ then $M(l,0)$ is a vertex operator algebra. However, $M(l,0)$ does not satisfy the $C_2$-condition. 
\item Let $I(l,0)$ be the unique maximal submodule of $M(l,0)$. Let $L(l,0)$ be the corresponding quotient. If $l\neq 0$ and $l\neq h\check{ }$ then $I(l,0)$ is an ideal of $M(l,0)$ and $L(l,0)$ is a simple vertex operator algebra. Moreover, $L(l,0)$ satisfies the $C_2$-condition.
\end{enumerate}
\end{prop}
\end{ex}

\smallskip

\begin{ex} {\bf Vertex Operator Algebras associated with Virasoro algebras}

We denote the Virasoro algebra by $\mathcal{L}=\oplus_{n\in\mathbb{Z}}\mathbb{C}L_n\oplus\mathbb{C}c$ with relation $$[L_m,L_n]=(m-n)L_{m+n}+\delta_{m+n,0}\frac{m^3-m}{12}C,~[L_m,C]=0.$$ We set $\mathfrak{b}=(\oplus_{n\geq 1}\mathbb{C}L_n)\oplus(\mathbb{C}L_0\oplus\mathbb{C}C)$. 

For any complex number $c$, let $\mathbb{C}_c$ be a 1-dimensional $\mathfrak{b}$-module defined as follows: $L_n\cdot 1=0$ for $n\geq 0$, $C\cdot 1=c\cdot 1$. We set $$V(c,0)=U(\mathcal{L})\otimes_{U(\mathfrak{b})}\mathbb{C}_c.$$ 
$V(c,0)$ has a unique maximal proper submodule $J(c,0)$. Let $L(c,0)$ be the unique irreducible quotient module of $V(c,0)$. It is well known that $\overline{V(c,0)}=V(c,0)/(U(\mathcal{L})L_{-1}1\otimes 1)$ is a vertex operator algebra. Moreover, 

\begin{prop}(\cite{DLM, FZ}) \ \ 

\begin{enumerate}
\item $L(c,0)$ is the unique irreducible quotient vertex operator algebra of $\overline{V(c,0)}$.
\item For $m\geq 2$, set $c_m=1-\frac{6}{m(m+1)}$. The vertex operator algebra $L(c_m,0)$ satisfies the $C_2$-condition. \end{enumerate}
\end{prop}
\end{ex}

\smallskip

\section{Mathieu-Zhao subspaces of Vertex Algebras}

In this section, we define Mathieu-Zhao subspaces of vertex algebras and study their properties. Also, we study relationships between Mathieu-Zhao subspaces of vertex algebras $V$ and Mathieu-Zhao subspaces of commutative associative algebras $V/C_2(V)$. 

\begin{dfn} Let $V$ be a vertex algebra. Let $M$ be a $\mathbb{C}$-subspace of $V$. We define the radical, left-strong radical, right-strong radcial, and strong radical of $M$ with respect to $0^{th}$ and $(-1)^{th}$-products in the following way. 
\begin{enumerate}
\item The radical of $M$ with respect to $0^{th}$ and $(-1)^{th}$-products is the set
\begin{eqnarray*}
r_{0,-1}(M)=\{v\in V&|&\text{ there exists }m\geq 0\text{ such that }v_{n_1}v_{n_2}...v_{n_t}v\in M\\
&&\text{ for all }t\geq m,~n_1,..., n_t\in\{0,-1\}\}.
\end{eqnarray*}
\item The left-strong radical of $M$ with respect to $0^{th}$ and $(-1)^{th}$-products is the set 
\begin{eqnarray*}
lsr_{0,-1}(M)=\{v\in V&|&\text{for }b\in V, \text{there exists }m\geq 0\text{ such that }\\
&&b_sv_{n_1}v_{n_2}...v_{n_t}v\in M\text{ for all }t\geq m,\\
&&s, n_1,..., n_t\in\{0,-1\}\}.
\end{eqnarray*}
\item The right-strong radical of $M$ with respect to $0^{th}$ and $(-1)^{th}$-products is the set 
\begin{eqnarray*}
rsr_{0,-1}(M)=\{v\in V&|&\text{for }w\in V, \text{there exists }m\geq 0\text{ such that }\\
&&(v_{n_1}...v_{n_t}v)_nw\in M\text{ for all }t\geq m,\\
&&n, n_1,..., n_t\in\{0,-1\}\}.
\end{eqnarray*}
\item The strong radical of $M$ with respect to $0^{th}$ and $(-1)^{th}$-products is the set
\begin{eqnarray*}
{sr}_{0,-1}(M)&=&lsr_{0,-1}(M)\cap rsr_{0,-1}(M).
\end{eqnarray*}

\end{enumerate}
\end{dfn}
\begin{rmk} Notice that 
\begin{enumerate}
\item $sr_{0,-1}(M)\subseteq lsr_{0,-1}(M)\subseteq r_{0,-1}(M)$.
\item $sr_{0,-1}(M)\subseteq rsr_{0,-1}(M)\subseteq r_{0,-1}(M)$.
\end{enumerate}
\end{rmk}

\smallskip

\begin{ex}(\cite{DW}) Let $V=\mathbb{C}[t,t^{-1}]$ be a commutative vertex algebra. Let $M\subset V$ be a set of Laurent polynomials whose constant term equals 0. Then $r_{0,-1}(M)=sr_{0,-1}(M)=t\mathbb{C}[t]\cup t^{-1}\mathbb{C}[t^{-1}]$. 
\end{ex}

\smallskip

\begin{lem} Let $M$ be a subspace of $V$. If $\mathcal{D}(V)\subset M$. Then $u\in lsr_{0,-1}(M)$ if and only if $u\in rsr_{0,-1}(M)$. In particular, we have $lsr_{0,-1}(M)=rsr_{0,-1}(M)=sr_{0,-1}(M)$. 
\end{lem}
\begin{proof} The above statements are consequences of equations (\ref{u0v}), (\ref{u-1v}).
\end{proof}

\smallskip

\begin{dfn} Let $V$ be a vertex algebra with a $\mathbb{C}$-subspace $M$. We say that $M$ is a Mathieu-Zhao subspace of $V$ with respect to $0^{th}$ and $(-1)^{th}$-products ($MZ_{0,-1}$-subspace of $V$) if $r_{0,-1}(M)=sr_{0,-1}(M)$.
\end{dfn}

\smallskip

\begin{ex} Let $I$ be an ideal of $V$. Then $I$ is a $MZ_{0,-1}$-subspace of $V$.

\begin{proof} Let $u\in r_{0,-1}(I)$. Then there exists $m\geq 0$ such that $u_{n_1}...u_{n_t}u\in I$ for all $t\geq m$ and $n_1,...,n_t\in\{0,-1\}$. Since $I$ is an ideal of $V$, this implies that for $b,w\in V$, $b_su_{n_1}...u_{n_t}u\in I$ and $(u_{n_1}...u_{n_t}u)_nw\in I$ for all $t\geq m$ and $s,n, n_1,...,n_t\in\{0,-1\}$. Hence, $u\in sr_{0,-1}(M)$. Moreover, $I$ is a $MZ_{0,-1}$-subspace of $V$. 
\end{proof}
\end{ex}

\smallskip

\noindent Recall that for $n\geq 2$, $C_n(V)=Span_{\mathbb{C}}\{u_{-n}v~|~u,v\in V\}$. Also, if $p\geq q$ then $C_p(V)\subseteq C_q(V)$.
\begin{ex} For $n\geq 2$, $C_n(V)$ is a $MZ_{0,-1}$-subspace of $V$.

\begin{proof} Let $u\in r_{0,-1}(C_n(V))$. Then there exists $m\geq 0$ such that $u_{n_1}...u_{n_t}u\in C_n(V)$ for all $t\geq m$ and $n_1,...,n_t\in\{0,-1\}$. By (\ref{commutativity}), (\ref{associativity}), we can conclude that for $v\in V$, $v_s(u_{n_1}...u_{n_t}u), (u_{n_1}...u_{n_t}u)_sv\in C_n(V)$ for all $s\in \{0,-1\}$. Hence, $u\in sr_{0,-1}(C_n(V))$. Moreover, $C_n(V)$ is a $MZ_{0,-1}$-subspace of $V$.
\end{proof}
\end{ex}

\begin{ex} Let $V$ be a vertex operator algebra of $CFT$-type. We set $V_+=\coprod_{n\geq 1}V_{(n)}$. Recall that $C_1(V)$ is a subspace of $V$ linearly spanned by elements of type $a_{-1}b$, $L(-1)w$ for $a,b\in V_+$, $w\in V$. Also, $\mathcal{D}(V)\subseteq C_2(V)\subseteq C_1(V)$. In this example, we will show that $C_1(V)$ is in fact a $MZ_{0,-1}$-subspace of $V$.

\begin{proof} Let $u\in r_{0,1}(C_1(V))$. Then there exists $m\geq 0$ such that $u_{n_1}...u_{n_t}u\in C_1(V)$ for all $t\geq m$ and $n_1,...,n_t\in\{0,-1\}$. Since $V$ is of $CFT$-type and $\mathcal{D}(V)\subseteq C_1(V)$, to show that $u\in sr_{0,-1}(C_1(V))$, we only need to show that for $v\in V_+$, 
\begin{equation}\label{C1MZ}v_s(u_{n_1}...u_{n_t}u)\in C_1(V)
\end{equation} for all $s\in \{0,-1\}$. However, (\ref{C1MZ}) follows immediately from (\ref{commutativity}), 
and (\ref{Drel}). We can conclude that $u\in sr_{0,-1}(C_1(V))$ and $C_1(V)$ is a $MZ_{0,-1}$-subspace of $V$.
\end{proof}
\end{ex}

\smallskip

\begin{rmk} We first recall the definition of Mathieu-Zhao subspace of associative algebra in \cite{Z5}. Let $R$ be a commutative unital ring, and let $\mathcal{A}$ be an associative unital $R$-algebra. Let $U$ be an $R$-submodule of $\mathcal{A}$. The radical of $U$, $r(U)$, and strong radical, $sr(U)$, are defined as follow:
\begin{eqnarray*}
&&r(U)=\{a\in A~|~\text{ there exists $m\geq 1$ such that }\underbrace{a\cdot...\cdot a}_{t\text{ terms }}\in U \text{ for all }t\geq m\},\\
&&sr(U)=\{a\in A~|~\text{ for $u,v\in A$, there exists $m\geq 1$ such that }b\underbrace{a\cdot...\cdot a}_{t\text{ terms }}c\in U\\
&&\hspace{3.5cm}\text{ for all }t\geq m\}.
\end{eqnarray*}
The subspace $U$ is a Mathieu-Zhao subspace of $\mathcal{A}$ when $r(U)=sr(U)$.

Next, we let $V$ be a commutative vertex algebra. Hence, $V$ is a commutative associative algebra. Let $M$ be a subspace of a commutative vertex algebra $V$. It is straight forward show that $r_{0,-1}(M)=r(M)$ and $sr_{0,-1}(M)=sr(M)$. Moreover $M$ is a $MZ_{0,-1}$-subspace of $V$ if and only if $M$ is a Mathieu-Zhao subspace of $V$ when we consider $V$ as a commutative associative algebra.
\end{rmk}

\smallskip

Next, we study properties of Mathieu-Zhao subspaces of vertex algebras.

\begin{prop} Let $V$ be a vertex algebra. Let $M$ be a subspace of $V$. 
\begin{enumerate}
\item If $r_{0,-1}(M)\subseteq r_{0,-1}(\{0\})$, then $M$ is a $MZ_{0,-1}$-subspace of $V$.
\item If $sr_{0,-1}(M)=V$ then $M$ is a $MZ_{0,-1}$-subspace of $V$.
\item If ${\bf 1}\in sr_{0,-1}(M)$ then $M$ is a $MZ_{0,-1}$-subspace of $V$.
\item If $U$ be is a subspace of $V$ such that $U\subseteq M$, then $r_{0,-1}(U)\subseteq r_{0,-1}(M)$.
\end{enumerate}
\end{prop}

\begin{proof} The above statements follow immediately from the definition of $MZ_{0,-1}$-subspace of $V$.
\end{proof}

\smallskip

\begin{thm} Let $V$ be a vertex algebra. Let $M$ is a subspace of $V$ such that ${\bf 1}\in M$. Then $M$ is a $MZ_{0,-1}$-subspace of $V$ if and only if $M=V$.
\end{thm}

\begin{proof} ($\Leftarrow$) Clear. 

($\Rightarrow$) Assume that $M$ is a $MZ_{0,-1}$-subspace of $V$. Since ${\bf 1}\in r_{0,-1}(M)=sr_{0,-1}(M)$, it implies that $b\in M$ for all $b\in V$. Hence, $V=M$.
\end{proof}

\smallskip

\noindent Recall that $f$ is a homomorphism from a vertex algebra $V^1$ to a vertex algebra $V^2$ if $f$ is a linear map such that $f(u_nv)=f(u)_nf(v)$ for $u,v\in V$, $n\in\mathbb{Z}$ and such that $f({\bf 1})={\bf 1}$. 
\begin{lem}\label{homMZ} Let $V^1$ and $V^2$ be vertex algebras and let $f:V^1\rightarrow V^2$ be a homomorphism. If $M$ is a $MZ_{0,-1}$-subspace of $V^2$, then $f^{-1}(M)$ is a $MZ_{0,-1}$-subspace of $V^1$.  
\end{lem}
\begin{proof} Let $v\in r_{0,-1}(f^{-1}(M))$. Hence, there exists $m\geq 0$ such that $$v_{n_1}v_{n_2}...v_{n_t}v\in f^{-1}(M)$$ for all $t\geq m,~n_1,..., n_t\in\{0,-1\}$. Using the fact that $f$ is a homomorphism, we can conclude that $$f(v)_{n_1}f(v)_{n_2}...f(v)_{n_t}f(v)\in M$$ for all $t\geq m,~n_1,..., n_t\in\{0,-1\}$. Hence, $f(v)\in r_{0,-1}(M)$. Let $b\in V^1$. Because $M$ is a $MZ_{0,-1}$-subspace of $V^2$, there exists $q\geq 0$ such that $$f(b)_sf(v)_{n_1}f(v)_{n_2}...f(v)_{n_t}f(v)\in M$$ and $(f(v)_{n_1}f(v)_{n_2}...f(v)_{n_t}f(v))_sf(b)\in M$ for all $t\geq q,~s, n_1,..., n_t\in\{0,-1\}$. Therefore, $v\in sr_{0,-1}(f^{-1}(M))$, and $f^{-1}(M)$ is a $MZ_{0,-1}$-subspace of $V^1$ as desired.  
\end{proof}

\smallskip

\begin{cor} Let $V^1$ and $V^2$ be vertex algebras such that $V^1\subseteq V^2$. If $M$ is a $MZ_{0,-1}$-subspace of $V^2$ then $M\cap V^1$ is a $MZ_{0,-1}$-subspace of $V^1$.
\end{cor}
\begin{proof} Let $\iota: V^1\rightarrow V^2$ be an inclusion map. Clearly, $\iota$ is a homomorphism from $V^1$ to $V^2$. By Lemma \ref{homMZ}, we can conclude immediately that $M\cap V^1=\iota^{-1}(M)$ is a $MZ_{0,-1}$-subspace of $V^1$.
\end{proof}

\smallskip

\begin{ex} Let $M$ be a $MZ_{0,-1}$-subspace of $V_L$. Then $M\cap M(1)$ is a $MZ_{0,-1}$-subspace of $M(1)$. 
\end{ex}

\smallskip

\begin{thm} Let $V$ be a vertex algebra and let $I$ be its ideal. Let $M$ be a subspace of $V$ such that $I\subseteq M$. Then 
$M$ is a $MZ_{0,-1}$-subspace of $V$ if and only if $M/I$ is a $MZ_{0,-1}$-subspace of $V/I$. 
\end{thm}
\begin{proof} ($\Leftarrow$) Let $f:V\rightarrow V/I$ be a projection map. By Lemma \ref{homMZ}, we can conclude immediately that if $M/I$ is a $MZ_{0,-1}$-subspace of $V/I$, then $M$ is a $MZ_{0,-1}$-subspace of $V$.

($\Rightarrow$) Now, assume that $M$ is a $MZ_{0,-1}$-subspace of $V$. Let $v+I\in r_{0,-1}(M/I)$. Then there exists $m\geq 0$ such that $$v_{n_1}v_{n_2}...v_{n_t}v+I\in M/I$$ for all $t\geq m,~n_1,..., n_t\in\{0,-1\}$. Consequently, we have $$v_{n_1}v_{n_2}...v_{n_t}v\in M$$ for all $t\geq m,~n_1,..., n_t\in\{0,-1\}$ and $v\in r_{0,-1}(M)$. Let $b\in V$. Because $r_{0,-1}(M)=sr_{0,-1}(M)$, there exists $q\geq 0$ such that $$b_sv_{n_1}v_{n_2}...v_{n_t}v\in M\text{ and }(v_{n_1}v_{n_2}...v_{n_t}v)_sb\in M$$ for all $t\geq q,~s, n_1,..., n_t\in\{0,-1\}$. Hence, $$b_sv_{n_1}v_{n_2}...v_{n_t}v+I\in M/I\text{ and }(v_{n_1}v_{n_2}...v_{n_t}v)_sb+I\in M/I.$$ Moreover, $v+I\in sr_{0,-1}(M/I)$ and $M/I$ is a $MZ_{0,-1}$-subspace of $V/I$. 
\end{proof}

\smallskip

\begin{ex} Let $U$ be a subspace of $M(l,0)$ such that $I(l,0)\subseteq U$. Then $U$ is a $MZ_{0,-1}$-subspace of $M(l,0)$ if and only if $U/I$ is a $MZ_{0,-1}$-subspace of $L(l,0)$.
\end{ex}

\smallskip

The following theorem will play an important role for the rest of this paper.

\begin{thm}\label{rel} Let $M$ be a subspace of $V$ such that $M\supseteq C_2(V)$. Let $a\in V$. We have the following:
\begin{enumerate}
\item $a\in r_{0,-1}(M)$ if and only if $a+C_2(V)\in r(M/C_2(V))$. 
\item $a\in sr_{0,-1}(M)$ if and only if $a+C_2(V)\in sr(M/C_2(V))$. 
\end{enumerate}
Hence, $M$ is a $MZ_{0,-1}$-subspace of $V$ if and only if $M/C_2(V)$ is a Mathieu-Zhao subspace of $V/C_2(V)$.
\end{thm}
\begin{proof} First, we will prove statement (1). Clearly, if $a\in r_{0,-1}(M)$ then $a+C_2(V)\in r(M/C_2(V))$. Now, we assume that $a+C_2(V)\in r(M/C_2(V))$. Hence, there exists $m\geq 0$ such that $\underbrace{a_{-1}....a_{-1}}_{t\text{ terms}}a+C_2(V)\in M/C_2(V)$ for all $t\geq m$. In particular, $\underbrace{a_{-1}....a_{-1}}_{t\text{ terms}}a\in M$. By Corollary \ref{a0-1a}, we can conclude immediately that $a\in r_{0,-1}(M)$.

Next, we will prove the statement (2). Clearly, if $u\in sr_{0,-1}(M)$ then for $b, w\in V$, there exists $p\geq 0$ such that $(b_{-1}w)_{-1}\underbrace{u_{-1}....u_{-1}}_{q\text{ terms}}u+C_2(V)\in M/C_2(V)$ for all $q\geq p$. This implies that 
$$b_{-1}\underbrace{u_{-1}...u_{-1}u_{-1}}_{q+1\text{ terms}}w\equiv b_{-1}(\underbrace{u_{-1}....u_{-1}}_{q\text{ terms}}u)_{-1}w+C_2(V)\in M/C_2(V)$$ for all $q\geq p$ and $u+C_2(V)\in sr(M/C_2(V))$. 

Now we will show that if $v+C_2(V)\in sr(M/C_2(V))$ then $v\in sr_{0,-1}(M)$. Let $v+C_2(V)\in sr(M/C_2(V)).$ To show that $v\in sr_{0,-1}(M)$, we only need to show that for $b\in V$, there exists $m\geq 0$ such that $b_0\underbrace{v_{-1}....v_{-1}}_{t\text{ terms}}v\in M$ for all $t\geq m$. Notice that 
\begin{eqnarray*}
b_0\underbrace{v_{-1}....v_{-1}}_{t\text{ terms}}v
&&=v_{-1}b_0\underbrace{v_{-1}....v_{-1}}_{t-1\text{ terms}}v+(b_0v)_{-1}\underbrace{v_{-1}....v_{-1}}_{t-1\text{ terms}}v\\
&&\equiv v_{-1}b_0\underbrace{v_{-1}....v_{-1}}_{t-1\text{ terms}}v+(\underbrace{v_{-1}....v_{-1}}_{t-1\text{ terms}}v)_{-1}(b_0v)\mod C_2(V)\\
&&\equiv (t+1)(\underbrace{v_{-1}....v_{-1}}_{t-1\text{ terms}}v)_{-1}(b_0v)\mod C_2(V).\\
\end{eqnarray*} 
Since $v+C_2(V)\in sr(M/C_2(V)$, this implies that there exists $d\geq 0$ such that $$(\underbrace{v_{-1}....v_{-1}}_{f\text{ terms}}v)_{-1}(b_0v)+C_2(V)\in M/C_2(V)$$ for all $f\geq d$. Therefore, $b_0\underbrace{v_{-1}....v_{-1}}_{f\text{ terms}}v\in M$ for all $f\geq d$ and $v\in sr_{0,-1}(M)$.

By using the statements (1) and (2), we can conclude immediately that $M$ is a $MZ_{0,-1}$-subspace of $V$ if and only if $M/C_2(V)$ is a Mathieu-Zhao subspace of $V/C_2(V)$.

\end{proof}

\smallskip 

\begin{cor} Let $V$ be a vertex algebra and let $M$ be a subspace of $V$ that contains $C_2(V)$ such that $M/C_2(V)$ is an ideal of a commutative associative algebra $V/C_2(V)$. Then $M$ is a $MZ_{0,-1}$-subspace of $V$.
\end{cor} 
\smallskip

\begin{ex} Let $\mathfrak{h}=\mathbb{C}\alpha$ and let $\lambda\in\mathbb{C}$. We set $$M=Span_{\mathbb{C}}\{u_{-1}(\alpha(-1){\bf 1}-\lambda{\bf 1})+v~|~u\in M(1),v\in C_2(M(1))\}.$$ Recall that a basis of $M(1)$ is $\{\alpha(-n_1).....\alpha(-n_k){\bf 1}~|~n_1\geq n_2\geq...\geq n_k\geq 1\}$, and $M(1)/C_2(M(1))$ is isomorphic to $\mathbb{C}[x]$ as commutative associative algebras. If $n_1\geq 2$, by \ref{associativity}, we have $(\alpha(-n_1).....\alpha(-n_k){\bf 1})_{-1}(\alpha(-1){\bf 1}-\lambda{\bf 1})\in C_2(V)$. Hence, $$M=Span_{\mathbb{C}}\{~(\underbrace{\alpha(-1)....\alpha(-1){\bf 1}}_{k\text{ terms}})_{-1}(\alpha(-1){\bf 1}-\lambda{\bf 1})+v~|~ k\geq 0, v\in C_2(V)\}$$ and $M/C_2(V)$ is isomorphic to an ideal generated by $x-\lambda$. Therefore, $M$ is a $MZ_{0,-1}$-subspace of $M(1)$. 
\end{ex}

\clearpage

\section{Mathieu-Zhao subspaces of $CFT$-type Vertex Operator Algebras that satisfy the $C_2$-cofiniteness condition}

In this section, we will study Mathieu-Zhao subspaces of vertex operator algebras $V$ of $CFT$-type that satisfy the $C_2$-cofiniteness condition. In particular, we will show that if $M$ is a subspace of $V$ such that $M\supseteq C_2(V)$ and ${\bf 1}\not\in M$, then $M$ is a $MZ_{0,-1}$-subspace of $V$. 

Also, at the end of this section we will discuss a relationship between  the Jacobian conjectures and Mathieu-Zhao subspaces of Heisenberg vertex operator algebras.

\smallskip

\begin{thm}\label{src2} Let $V$ be a vertex operator algebra. Let $M$ be a subspace of $V$. 

\begin{enumerate}
\item If $V=\oplus_{n=n_0}^{\infty}V_n$ and $n_0<0$, then $sr_{0,-1}(M)\neq \{0\}$. In fact, if $v\in \oplus_{n=n_0}^{-1}V_n$, then $v\in sr_{0,-1}(M)$.
\item Assume that $V$ satisfies the $C_2$-condition and $M\supseteq C_2(V)$. If $u\in \oplus_{i=1}^kV_k$ for some positive integer $k$, then $u\in sr_{0,-1}(M)$. 
\end{enumerate}
\end{thm}
\begin{proof} To proof the first statement, we observe that if $v\in \oplus_{n=n_0}^{-1}V_n$, then there exists a nonnegative integer $m$ such that $v_{n_1}v_{n_2}...v_{n_t}v=0$ for all $n_i\in\{0,-1\}$, $m\geq t$. This implies that $b_s(v_{n_1}v_{n_2}...v_{n_t}v)=(v_{n_1}v_{n_2}...v_{n_t}v)_sb=0$ for all $b\in V$, $s\in\{0,-1\}$. Hence, $v\in sr_{0,-1}(M)$.

The second statement follows from the following facts. If $V$ satisfies the $C_2$-cofiniteness condition and $u\in \oplus_{i=1}^kV_k$ for some positive integer $k$, then there exists a nonnegative integer $m$ such that $u_{n_1}u_{n_2}...u_{n_t}u\in C_2(V)$ for all $n_i\in\{0,-1\}$, $t\geq m$. By using the facts that $\mathcal{D}(V)\subseteq C_2(V)\subseteq M$ and $C_2(V)$ is closed under $(0)^{th}$ and $(-1)^{th}$-products, we can conclude that $u\in sr_{0,-1}(M)$. 
\end{proof}

\smallskip

 Next, we recall useful information about Mathieu Zhao subspaces of associative algebras that we will need to use later on.
\begin{prop}\label{Zhao}\cite{Z5}
Let $K$ be an arbitrary field and $\mathcal{A}$ an associative algebra over $K$. For convenience, we denote by $\mathcal{G}(\mathcal{A})$ the collection of all $K$-subspaces $\mathcal{V}$ of $\mathcal{A}$ such that $r(\mathcal{V})$ is algebraic over $K$.
\begin{enumerate}
\item Let $\mathcal{V}$ be a $K$-subspace of $\mathcal{A}$. Then $\mathcal{V}\in\mathcal{G}(\mathcal{A})$ if one of the following four conditions holds:
\begin{enumerate}
\item $\mathcal{A}$ is algebraic over $K$;
\item $\mathcal{V}$ is algebraic over $K$;
\item $dim_K\mathcal{A}<\infty$;
\item $dim_K\mathcal{V}<\infty$. 
\end{enumerate}
\item Let $\mathcal{V}\in\mathcal{G}(\mathcal{A})$ such that $\mathcal{V}$ does not contain any nonzero idempotent. Then $\mathcal{V}$ is a Mathieu subspace of $\mathcal{A}$. 
\item Let $\mathcal{A}$ be a commutative $K$-algebra and $\mathcal{V}\in \mathcal{G}(\mathcal{A})$. Then $\mathcal{V}$ is a Mathieu subspace of $\mathcal{A}$ if and only if $r(\mathcal{V})$ is an ideal of $\mathcal{A}$. 
\end{enumerate}
\end{prop}

\smallskip

\begin{thm} Let $V$ be a vertex algebra. Let $M$ be a subspace of $V$ such that $C_2(V)\subseteq M$. Assume that $dim M/C_2(V)<\infty$. Then $M$ is a $MZ_{0,-1}$-subspace of $V$ if and only if $r(M/C_2(V))$ is an ideal of $V/C_2(V)$.
\end{thm}
\begin{proof} This follows immediately from Theorem \ref{rel} and Propositin \ref{Zhao}.
\end{proof}

\smallskip

\noindent Now we study Mathieu-Zhao subspaces of $CFT$-type vertex operator algebras that satisfy the $C_2$-cofiniteness condition.
\begin{lem} Let $V$ be a vertex operator algebra of $CFT$-type that satisfies the $C_2$-cofiniteness condition. Then $V/C_2(V)$ has no nontrivial idempotent.
\end{lem}
\begin{proof} Clearly if $u\in \oplus_{n=1}^{\infty}V_n$ then $u+C_2(V)$ is nilpotent in $V/C_2(V)$. Let $v\in \oplus_{n=1}^{\infty}V_n$. If $(1+v)_{-1}(1+v)\equiv 1+v\mod C_2(V)$ then $v_{-1}v\equiv -v\mod C_2(V)$. Moreover, we have $\underbrace{v_{-1}....v_{-1}v}_{n\text{ terms}}\equiv (-1)^{n-1}v\mod C_2(V)$ for all $n\geq 2$. Hence, $v\equiv0\mod C_2(V)$, and $V/C_2(V)$ has no nontrivial idempotent.
\end{proof}

\begin{thm} Let $V$ be a vertex algebra of $CFT$-type that satisfies the $C_2$-cofiniteness condition. Let $M$ be a subspace of $V$ such that ${\bf 1}\not\in M$ and $C_2(V)$ is a subset of $M$, then $M$ is a $MZ_{0,-1}$-subspace of $V$ and $r_{0,-1}(M)=\oplus_{n=1}^{\infty}V_n$. \end{thm}

\begin{proof} Since $V/C_2(V)$ has no nontrivial idempotent, $dim_{\mathbb{C}} ~M/C_2(V)<\infty$ and ${\bf 1}+C_2(V)\not\in M/C_2(V)$, by Proposition \ref{Zhao}, we can conclude that $M/C_2(V)$ is an MZ-subspace of $V/C_2(V)$. Also, by Theorem \ref{rel}, $M$ is a $MZ_{0,-1}$-subspace of $V$.

Recall that $\oplus_{n=1}^{\infty}V_i\subseteq sr_{0,-1}(M)=r_{0,-1}(M)$. Let $u\in V$ such that $u={\bf 1}+v$ where $v\in \oplus_{n=1}^{\infty}V_n$. If $u\in r_{0,-1}(M)$ then $u+C_2(V)\in r(M/C_2(V))$. Since $r(M/C_2(V))$ is an ideal, and $v+C_2(V),~u+C_2(V)$ are in $r(M/C_2(V))$, we can conclude that ${\bf 1}+C_2(V)\in r(M/C_2(V))$. This implies that ${\bf 1}\in M$ which is impossible. Hence elements of the form ${\bf 1}+v$ where $v\in \oplus_{n=1}^{\infty}V_n$ are not in $r_{0,-1}(M)$. Therefore, $r_{0,-1}(M)=\oplus_{n=1}^{\infty}V_n$.  
\end{proof} 

\smallskip

\begin{ex} Assume that $V$ is either $V_L$ or $L(l,0)$ or $L(c_m,0)$ ($m\geq 2$). If $M$ is a subspace of $V$ such that ${\bf 1}\not\in M$ and $C_2(V)$ is a subset of $M$, then $M$ is a $MZ_{0,-1}$ subspace of $V$.
\end{ex}

\begin{rmk} If $V$ is a vertex operator algebra of $CFT$-type that satisfies the $C_2$-condition. If $M$ is a subspace of $V$ such that $C_2(V)\subseteq M$ and $M\neq V$, then $M$ is a $MZ_{0,-1}$ -subspace of $V$ only when ${\bf 1}\not\in M$.
\end{rmk}

\smallskip

We end this paper with a discussion about a connection between the Jacobian conjecture and Mathieu-Zhao subspaces of Heisenberg vertex operator algebras. The Jacobian conjecture was originally formulated in an article written by Keller in 1939 for the occasion of Furtw$\ddot{a}$ngler's 70th birthday \cite{K}. In this article he described a problem about polynomial automorphisms which he considered as interesting problem to think about. Nowadays, this problem is called the Keller problem. The Keller problem has been generalized to the Jacobian conjecture which is named after the occurrence of the Jacobian determinant in the problem. 

\medskip

\noindent{\bf Jacobian Conjecture.} For every positive integer $n$ the following statement holds: Every $\mathbb{C}$-endomorphism $\phi$ of $\mathbb{C}[x_1,...,x_n]$ with $\det J\phi=1$ is an automorphism, where $J\phi=\left(\frac{\partial \phi(x_i)}{\partial x_j}\right)_{1\leq i,j\leq n}$. 

\medskip

\noindent The Jacobian conjecture is a problem that is still open in all dimension $n>1$. There are partial solutions of this problem \cite{BE1, BE2, D, H, Mo, W, Wr}.

As we mentioned in the introduction, Mathieu-Zhao subspaces of commutative associative algebras grew out of several attempts to solve the Jacobian conjecture. One of these attempts was a result due to Zhao \cite{Z1}. He stated the following conjecture and showed that this conjecture implies the Jacobian conjecture. 

\medskip

\noindent{\bf Image Conjecture.}  $\sum_{i=1}^n (\partial_{x_i}-\zeta_i)\mathbb{C}[\zeta_1,...,\zeta_n,x_1,...,x_n]$ is a Mathieu-Zhao subspace of $\mathbb{C}[\zeta_1,...,\zeta_n,x_1,...,x_n].$

\medskip

Now, we let $\mathfrak{h}$ be a vector space over $\mathbb{C}$ with a basis $\{\beta^1,...,\beta^{n}, \alpha^1,...,\alpha^n\}$ and a bilinear form $(~,~)$ such that $(\beta^i,\beta^j)=\delta_{i,j}$, $(\alpha^i,\alpha^j)=\delta_{i,j}$, $(\beta^i,\alpha^j)=0$. Recall that
$M(1)/C_2(M(1))$ is isomorphic to $\mathbb{C}[\zeta_1,...,\zeta_n,x_1,....,x_n]$ as commutative associative algebras. 
For $i\in\{1,...,n\}$, we define a linear map $f_i:M(1)\rightarrow M(1)$ by $f_{i}(a)=(\beta^i_1-\alpha^i_{-1})(a)$ for $a\in M(1)$. Clearly, $f_i(C_2(M(1)))\subseteq C_2(M(1))$.

By Theorem \ref{rel}, we can conclude that 
\begin{thm} $\sum_{i=1}^k f_i(M(1))+C_2(M(1))$ is a $MZ_{0,-1}$-subspace of $M(1)$ if and only if $\sum_{i=1}^k (\partial_{x_i}-\zeta_i)\mathbb{C}[\zeta_1,...,\zeta_n,x_1,...,x_n]$ is a Mathieu-Zhao subspace of $\mathbb{C}[\zeta_1,...,\zeta_n,x_1,...,x_n].$
\end{thm}
\begin{cor} If $\sum_{i=1}^k f_i(M(1))+C_2(M(1))$ is a $MZ_{0,-1}$-subspace of $M(1)$ for all $n\in\mathbb{N}$ then for every positive integer $n$ the following statement holds: Every $\mathbb{C}$-endomorphism $\phi$ of $\mathbb{C}[x_1,...,x_n]$ with $\det J\phi=1$ is an automorphism, where $J\phi=\left(\frac{\partial \phi(x_i)}{\partial x_j}\right)_{1\leq i,j\leq n}$.
\end{cor}

\clearpage

\section{References}


\end{document}